\documentclass[12pt]{amsart}

\usepackage{amsfonts,amssymb}
\usepackage{graphicx}
\usepackage{epsfig}
\usepackage{latexsym}
\usepackage{amsmath,amsthm}
\usepackage{mathrsfs}
\usepackage[all,cmtip]{xy}
\usepackage[pagebackref=true]{hyperref}
\hypersetup{   pdfborder={0 0 0},   
              colorlinks = false}
\usepackage{tikz-cd}
\usepackage{xcolor}
\renewcommand*{\backref}[1]{}
\renewcommand*{\backrefalt}[4]{%
  \ifcase #1 %
    No citations.% use \relax if you do not want the "No citations" message
  \or
    ($\uparrow$ #4).%
  \else
    ($\uparrow$ #4).%
  \fi%
}
\theoremstyle{plain}
\newtheorem{theorem}{Theorem}[section]
\newtheorem{lemma}[theorem]{Lemma}
\newtheorem{definition}[theorem]{Definition}
\newtheorem{proposition}[theorem]{Proposition}

\newtheorem{remark}[theorem]{Remark}
\newtheorem*{example}{Example}

\numberwithin{equation}{section}

\newcommand{\bigslant}[2]{{\raisebox{.2em}{$#1$}\left/\raisebox{-.2em}{$#2$}\right.}}

\makeatletter
\def\moverlay{\mathpalette\mov@rlay}
\def\mov@rlay#1#2{\leavevmode\vtop{%
   \baselineskip\z@skip \lineskiplimit-\maxdimen
   \ialign{\hfil$\m@th#1##$\hfil\cr#2\crcr}}}
\newcommand{\charfusion}[3][\mathord]{
    #1{\ifx#1\mathop\vphantom{#2}\fi
        \mathpalette\mov@rlay{#2\cr#3}
      }
    \ifx#1\mathop\expandafter\displaylimits\fi}
\makeatother

\newcommand\Overline[2][1pt]{%
    \begin{tikzpicture}[baseline=(a.base)]
      \node[inner xsep=0pt,inner ysep=1.5pt] (a) {$#2$};
      \draw[line width= #1] (a.north west) -- (a.north east);
    \end{tikzpicture}
    }
\newcommand{\la}[1]{\mathcal{#1}}

\newcommand{\C}{\mathbb{C}}
\newcommand{\R}{\mathbb{R}}

\hypersetup{colorlinks, citecolor=blue, filecolor=black, linkcolor=red}
\newcommand{\inv}[1]{\ensuremath{\sigma_{#1}}}
\newcommand{\invs}[1]{\ensuremath{\sigma^\star_{#1}}}

\newcommand{\rst}[1]{\ensuremath{\sigma^{#1}}}

\begin{document}
\baselineskip=15.5pt

\title{Moduli of Real (res. Quaternionic) $\la{L}$-connections}
%\author{Sanjay~Amrutiya}
%\address{Department of Mathematics, IIT Gandhinagar,
% Near Village Palaj, Gandhinagar - 382355, India}
% \email{samrutiya@iitgn.ac.in}
\author{Ayush~Jaiswal}
\address{Department of Mathematics, IISER Tirupati,
 Srinivasapuram-Jangalapalli Village, Panguru (G.P) Yerpedu Mandal, Tirupati - 517619, Chitoor District, Andhra Pradesh India.}
\email{ayushjwl.math@gmail.com}
%\thanks{SA was supported by the SERB-DST under project no. YSS/2015/001182 and MTR/2018/000475.}
\thanks{}
\date{}

%%%%%%%%%%%%%%%%%%%%%%%%%%%%%   Main Document   %%%%%%%%%%%%%%%%%%%%%%%%%%%%%%%%%%
\begin{abstract}
We have studied irreducible real (respectively, quaternionic) Lie algebroid connections and prove that the Gauge theoretic moduli space has Hausdorff Hilbert manifold structure. This work generalises some known results about simple semi-connections for complex vector bundle on compact complex manifold in real algebraic geometry.
\end{abstract}
\maketitle
\section{Introduction}\label{intro}
Gauge theoretic moduli spaces are useful in computing differential geometric invariants like Seiberg-Witten invariants, Donaldson invariant and many more. Gauge theoretic analytic tools are useful in the theory of four-dimensional manifold as well as Yang-Mills theory, etc.
 
In this article, we have studied Lie algebroid connections with real as well as quaternionic structures. The goal of the article is to give description of moduli of real as well as quaternionic Lie algebroid connections for complex vector bundle on compact complex manifold and describe the natural structure of locally Hausdorff Hilbert manifold with Symplectic structure on it.

In \cite{MFA}, Atiyah has described real (respectively, quaternionic) vector bundles as a complex vector bundle on a space with involution, say $(X,\inv{X})$ along with the bundle has a $\C$-antilinear automorphism $\rst{E}$ over the involution $\inv{X}$ such that $\rst{E}\circ \rst{E}\equiv \mathrm{id}_E$ (respectively, $\rst{E}\circ \rst{E}\equiv -\mathrm{id}_E$).

In section \ref{sec-3}, we have studied the real structure on the space of $\la{L}$-connections and proved that two $\la{L}$-connection lie in same gauge orbit iff they lie in same real gauge orbit. 

In section \ref{sec-4}, we have described the moduli of real (respectively, quaternionic) Sobolev irreducible Lie algebroid connections for complex vector bundle on complex manifold with real (respectively, quaternionic) structure and described the natural Hausdorff Hilbert manifold structure on it. Along with, we have proved $p:\widehat{A}(E,\la{L})_l^{\inv{}}\rightarrow \widehat{B}(E,\la{L})^{\inv{}}_l$ is a principal $\mathrm{Gau}(E)^{\inv{}}_{l+1}$-bundle, where $\widehat{A}(E,\la{L})_l^{\inv{}}$ is the space of irreducible real (respectively, quaternionic) Sobolev $\la{L}$-connections and $\widehat{B}(E,\la{L})^{\inv{}}_l$ is gauge equivalence classes. This is generalization of results in \cite{MLCO} and \cite{SK}.

\section{Preliminaries}\label{sec-2}
\begin{definition}\rm{
A complex (respectively, real) Lie algebroid of rank $r$ is a triple $(\la{L},[.,.],\sharp)$ of a complex (respectively, real) vector bundle of rank $r$, $\la{L}$ over a connected complex manifold $X$, a vector space morphism $[.,.]:\Gamma(X,\la{L})\times \Gamma(X,\la{L})\rightarrow \Gamma(X,\la{L})$ also known as Lie bracket structure on the space of global sections and a vector bundle morphism $\sharp:\la{L}\rightarrow TX_\C$ (respectively, $\sharp:\la{L}\rightarrow TX$), known as the anchor map such that the diagram commutes,

\[\begin{tikzcd}
	{\mathcal{L}} & TX && {\mathcal{L}} & {TX_\mathbb{C}} \\
	X & X && X & X
	\arrow["\sharp", from=1-1, to=1-2]
	\arrow["{\pi_\mathcal{L}}"', from=1-1, to=2-1]
	\arrow["{\pi_{TX}}", from=1-2, to=2-2]
	\arrow["\sharp", from=1-4, to=1-5]
	\arrow["{\pi_\mathcal{L}}"', from=1-4, to=2-4]
	\arrow["{\pi_{TX_\mathbb{C}}}", from=1-5, to=2-5]
	\arrow["\mathrm{id}_X"', from=2-1, to=2-2]
	\arrow["\mathrm{id}_X"', from=2-4, to=2-5]
\end{tikzcd}\]

the achor map satisfies the following conditions:
\begin{enumerate}
\item $\sharp\big([s_1,s_2]\big)=[\sharp(s_1),\sharp(s_2)]$
\item $[s_1,f.s_2]=f[s_1,s_2]+\big(\sharp(s_1)(f)\big)s_2$.
\end{enumerate}
}
\end{definition}

\begin{example}\rm{ 
Some examples of Lie algebroids are following:
\begin{enumerate}
\item For a given manifold, the tangent bundle with canonical Lie bracket structure on space of global sections, in this example the anchor map is the identity map.
\item A Lie algebra is a Lie algebroid over a point manifold with anchor map, a trivial map.
\item A subbundle of the tangent bundle of a given manifold, closed under Lie bracket structure, is a Lie algebroid.
\item Any bundle with fiber as a Lie algebra with anchor map, a trivial map is a Lie algebroid.
\end{enumerate}
}
\end{example}
\subsection{Lie algebroid connections}
For a given complex Lie algebroid $(\la{L},[.,.],\sharp)$ of rank $r$, the vector space 
$$
\Gamma(X,\Lambda^\bullet \la{L}^\star)=\bigoplus_{k}\Gamma(X,\Lambda^k\la{L}^\star)
$$
has a graded algebra structure induced by the exterior product map such that
%$$
%\wedge:\Gamma(X,\Lambda^k\la{L}^\star\otimes E)\times \Gamma(X,\Lambda^l\la{L}^\star\otimes E)\rightarrow \Gamma(X,\Lambda^{k+l}\la{L}^\star\otimes E)
%$$
for $\phi\in \Gamma(X,\Lambda^k\la{L}^\star)$ and $\psi\in \Gamma(X,\Lambda^l\la{L}^\star)$, $\phi\wedge \psi\in \Gamma(X,\Lambda^{k+l}\la{L}^\star)$ is given by 
$$
(\phi\wedge \psi)(\xi_1,\xi_2,\dots,\xi_{l+k})=\frac{1}{k!}\frac{1}{l!}\displaystyle\sum_{\sigma\in S_{l+k}}\phi\big(\xi_{\sigma(1)},\xi_{\sigma(2)},\dots,\xi_{\sigma(l)}\big).\psi\big(\xi_{\sigma(l+1)},\xi_{\sigma(l+2)},\dots,\xi_{\sigma(l+k)}\big).
$$

%Define a degree 1 operator $d_{\mathcal{CL}}$ on graded algebra $\Gamma(X,\bigwedge^\bullet \la{L}^\star)$ described as follows.

For $\phi\in \Gamma(X,\Lambda^k \la{L}^\star)$ and $\xi_l\in \Gamma(X,\la{L})$ $(0\leq l\leq k)$, define an operator $d_\la{L}:\Gamma(X,\Lambda^k \la{L}^\star)\rightarrow \Gamma(X,\Lambda^{k+1} \la{L}^\star)$ such that
\begin{align*}
d_\la{L}(\phi)(\xi_0,\dots,\xi_k)&=\displaystyle\sum_{n=0}^{k} (-1)^n \sharp(\xi_n)\phi\big(\xi_0,\xi_1,\dots,\widehat{\xi_n},\dots,\xi_k\big)\\
&\qquad+\displaystyle\sum_{l<m}(-1)^{l+m}\phi\big([\xi_l,\xi_m],\xi_0,\dots,\widehat{\xi_l},\dots,\widehat{\xi_m},\dots,\xi_n\big).
\end{align*}

The operator $d_\la{L}$ can be seen as a generalization of the de-Rham operator $d$ satisfying the property $d_\la{L}\circ d_\la{L}\equiv 0$. For a given element $\xi\in \Gamma(X,\la{L})$ the Lie derivative operator $\mathfrak{L}_\xi^\la{L}:\Gamma(X,\Lambda^k \la{L}^\star)\rightarrow \Gamma(X,\Lambda^k \la{L}^\star)$ is defined as 
$$
\mathfrak{L}_\xi^\la{L}(\phi)(\xi_1,\xi_2\dots,\xi_k)=\sharp(\xi)\phi(\xi_1,\xi_2,\dots,\xi_k)-\displaystyle\sum_{j=1}^n\phi(\xi_1,\xi_2,\dots,[\xi,\xi_j],\dots,\xi_k)
$$
and the degree -1, insertion operator $i^\la{L}_\xi:\Gamma(X,\Lambda^{k+1} \la{L}^\star)\rightarrow \Gamma(X,\Lambda^k \la{L}^\star)$ is defined as
$$
i^\la{L}_\xi(\phi)(\xi_1,\xi_2,\dots,\xi_k)=\phi(\xi,\xi_1,\xi_2,\dots,\xi_k).
$$
\begin{remark}\rm{
For simplicity, the space of global sections of the bundle $\Lambda^k\la{L}^\star$ will be denoted by $\mathcal{A}^k_{\la{L}}(X)$ and the graded algebra $\Gamma(X,\Lambda^\bullet \la{L}^\star)$ will be denoted by $\mathcal{A}^\bullet_{\la{L}}(X)$.
}
\end{remark}

Let $(\la{L},[.,.],\sharp)$ (respectively, $E$) be a complex Lie algebroid (respectively, complex vector bundel) on a connected complex manifold $X$. The space of global sections of bundle $\Lambda^kL^\star \otimes E$ will be denoted by $\mathcal{A}_{\la{L}}^k(X,E)$, whose elements are called global $\la{L}$-forms of type $k$ taking values in the bundle $E$. 
\begin{definition}\rm{
A $\C$-linear map 
$$
\nabla:\mathcal{A}^0_{\la{L}}(X,E)\rightarrow \mathcal{A}^1_{\la{L}}(X,E)
$$
satisfying the following Leibnitz type identity
$$
\nabla(fs)=d_\la{L}(f)s+f\nabla(s)
$$
for $f\in C^\infty(X)$ and $s\in \mathcal{A}^0_{\la{L}}(X,E)$, is called an $\la{L}$-connection.
}
\end{definition}
%\begin{remark}
%{\color{red}A $d$-smooth lie algebroid connection $nabla^E$ for a given $d$-smooth vector $E$, can be expressed as a family of smooth lie algebroid connections satisfying appropriate compatibility condition}
%\end{remark}
\begin{definition}\rm{
For a given element $\xi\in \Gamma(X,\la{L})$, we have an $\C$-linear map
$$
\nabla_\xi:\mathcal{A}^0_{\la{L}}(X,E)\rightarrow \mathcal{A}^0_{\la{L}}(X,E)\quad\Big(\nabla_\xi(fs)=\mathfrak{L}_\xi^\la{L}(f) s+f \nabla(s)(\xi)\text{ for }f\in C^\infty(X),s\in \mathcal{A}^0_\la{L}(X,E)\Big)
$$
The element $\nabla_\xi(s)\in \mathcal{A}_{\la{L}}^1(X,E)$ is called co-variant $\la{L}$-derivative of $s\in \mathcal{A}^0_{\la{L}}(X,E)$ and the operator $\nabla_\xi$ is called co-variant $\la{L}$-differential operator, in the direction of $\xi\in \Gamma(X,\la{L})$.
}
\end{definition}
\begin{theorem}
For a given $\la{L}$-connection $\nabla$, there is a degree 1 differential operator $d^\nabla:\mathcal{A}^\bullet_\la{L}(X,E)\rightarrow \mathcal{A}^\bullet_\la{L}(X,E)$ such that $d^\nabla|_{\mathcal{A}^0_\la{L}(X,E)}\equiv \nabla$, and
\begin{enumerate}
\item $d^\nabla(\alpha\wedge \phi)=d_\la{L}(\alpha)\wedge \phi+(-1)^l\alpha\wedge d^\nabla(\phi)$
\item $d^\nabla|_{\mathcal{A}^0_\la{L}(X,E)}\equiv \nabla$ and $(d^\nabla(s))(\psi)=\nabla_\psi(s)$
\end{enumerate}
for $\alpha\in \mathcal{A}^l_\la{L}(X),\phi\in \mathcal{A}^\bullet_\la{L}(X,E),s\in \mathcal{A}^0_\la{L}(X,E)$ and $\psi\in \Gamma(X,\la{L})$. 
\end{theorem}
\begin{proof}
The proof follows in the same line of arguements as in the case of affine connections modeled on the vector space $\Gamma(X,T^\star X\otimes\mathrm{End}(E))$ (see, \cite{SK}).
\end{proof}
\begin{lemma}\label{liealgconn}
The space of $\la{L}$-connections $A(E,\la{L})$, is an affine space modeled on the vector space $\mathcal{A}^1_\la{L}(X,\textrm{End}(E))$.
\end{lemma}
\begin{proof}
Let $\{U_\alpha,\phi_\alpha\}_{\alpha\in I}$ be a trivialization of the bundle $E$ along with locally finite open covering $\{U_\alpha\}_{\alpha\in I}$ of $X$ and $\{g_\alpha\}_{\alpha\in I}$ be the smooth partition of unity subordinate to the open covering. Take $s\in \mathcal{A}^0_\la{L}(X,E)$, then $\mathrm{Supp}(g_\alpha s)\subset U_\alpha$ and $\phi_\alpha(g_\alpha s|_{U_\alpha})\in \Gamma(U_\alpha,U_\alpha\times \C^r)$ along with $\mathrm{Supp}(\phi_\alpha(g_\alpha s|_{U_\alpha}))\subset U_\alpha$ for each $\alpha\in I$, therefore $\phi_\alpha(g_\alpha s|_{U_\alpha})$ can be extended to a global section, say $\phi_\alpha(g_\alpha s|_{U_\alpha})\subset \Gamma(X,X \times \C^r)$.

Let $\nabla_0$ be the trivial $\la{L}$-connection for the trivial bundle on $X$ that is $\nabla_0(f\times s)=d_\la{L}(f)\times s$ for $f\in C^\infty(X)$ and $s\in \Gamma(X,X\times \C^r)$, we have $\mathrm{Supp}(\nabla_0(\phi_\alpha(g_\alpha s|_{U_\alpha}))\subset U_\alpha$ for each $\alpha\in U_\alpha$, furthermore 
$$
\left[\mathrm{id}_{{\la{L}^\star}|_{U_\alpha}}\otimes \phi^{-1}_\alpha\circ \nabla_0(\phi_\alpha(g_\alpha s|_{U_\alpha}))\right]\in \mathcal{A}^1_\la{L}(U_\alpha,E)
$$
Define
$$
\nabla(s)(x)=\displaystyle \sum_\alpha \mathrm{id}_{{\la{L}^\star}|_{U_\alpha}}\otimes \phi^{-1}_\alpha\circ \nabla_0(\phi_\alpha(g_\alpha s|_{U_\alpha}))(x)
$$
which will be finite sum for a given $x\in X$, because of local finiteness of the covering.

The map $\nabla:\mathcal{A}^0_\la{L}(X,E)\rightarrow \mathcal{A}^1_\la{L}(X,E)$ is an $\la{L}$-connection, because
\begin{align*}
\nabla(fs)(x)&=\displaystyle \sum_\alpha \mathrm{id}_{{\la{L}^\star}|_{U_\alpha}}\otimes \phi^{-1}_\alpha\circ \nabla_0(\phi_\alpha(g_\alpha (fs)|_{U_\alpha}))(x)\\
&=\displaystyle \sum_\alpha \mathrm{id}_{{\la{L}^\star}|_{U_\alpha}}\otimes \phi^{-1}_\alpha\circ \nabla_0(f|_{U_\alpha}\phi_\alpha\circ (g_\alpha \circ s)|_{U_\alpha})(x)\\
&=\displaystyle \sum_\alpha \mathrm{id}_{{\la{L}^\star}|_{U_\alpha}}\otimes \phi^{-1}_\alpha\circ \big[d_\la{L}(f|_{U_\alpha})(x)\otimes\phi_\alpha\circ(g_\alpha \circ s)|_{U_\alpha}(x)\\
&\qquad\qquad+f|_{U_\alpha}(x)\otimes \nabla_0(\phi_\alpha\circ(g_\alpha \circ s|_{U_\alpha}))(x)\big]\\
&=d_\la{L}(f)(x)\otimes\displaystyle \sum_\alpha (g_\alpha \circ s)|_{U_\alpha}(x)\\
&\qquad\qquad+f(x)\otimes \displaystyle \sum_\alpha \mathrm{id}_{{\la{L}^\star}|_{U_\alpha}}\otimes \phi^{-1}_\alpha\circ\nabla_0(\phi_\alpha\circ(g_\alpha \circ s|_{U_\alpha}))(x)\big]\\
&=d_\la{L}(f)(x)\otimes s(x)+f(x)\otimes \nabla(s)(x)
\end{align*}
Only thing, we need to check is, for two $\la{L}$-connections $\nabla_i$ $(i=1,2)$, $(\nabla_1-\nabla_2)\in \mathcal{A}^1_\la{L}(X,\mathrm{End}(E))$, which is obvious because 
\begin{align*}
(\nabla_1-\nabla_2)(fs)&=(d_\la{L}(f)\otimes s+f\nabla_1(s))-(d_\la{L}(f)\otimes s+f\nabla_2(s))\\
&=f(\nabla_1-\nabla_2)(s)
\end{align*}
That is $\nabla_1-\nabla_2\in \mathcal{A}^1_\la{L}(X,\mathrm{End}(E))$. Hence, we have the following remark.
\end{proof}
\begin{remark}\rm{
Let $\nabla_0$ be the trivial $\la{L}$-connection, any Lie algebroid connection $\nabla\in A(E,\la{L})$ can be expressed as $\nabla=\nabla_0+\alpha$, where $\alpha\in \mathcal{A}^1_\la{L}(X,\mathrm{End}(E))$. Therefore, the space of $\la{L}$-connections can be expressed as,
\begin{equation}
A(E,\la{L})=\{\nabla_0+\alpha\mid \alpha\in \mathcal{A}^1_\la{L}(X,\mathrm{End}(E))\}
\end{equation}
}
\end{remark}
Like an affine connection, an $\la{L}$-connection $\nabla^E$ for a bundle $E$ induces an $\la{L}$-connection $\nabla^{\mathrm{End}(E)}$ for the bundle $\mathrm{End}(E)$ such that 
\begin{equation}
\nabla^E(Ts)=(\nabla^{\mathrm{End}(E)}T)(S)+T(\nabla^E(S))
\end{equation}
for $T\in \Gamma(X,\mathrm{End}(E))$ and $s\in \Gamma(X,E)$.

The graded $\C$-vector space $\mathcal{A}^\bullet_\la{L}(X,\mathrm{End}(E))$ has a canonical structure of associate algebra such that for $\omega\in \mathcal{A}^p_\la{L}(X,\mathrm{End}(E))$ and $\tau\in \mathcal{A}^q_\la{L}(X,\mathrm{End}(E))$, $\omega\wedge \tau$ is given by 
$$
\omega\wedge\tau(\xi_1,\xi_2,\dots,\xi_{p+q})=\frac{1}{p!q!}\displaystyle\sum_{\sigma\in S(p+q)} Sign(\sigma)\omega(\xi_{\sigma(1)},\dots,\xi_{\sigma(p)})\tau(\xi_{\sigma(p+1},\dots,\xi_{\sigma(p+q)})
$$
as well as a Lie algebra structure with the Lie bracket $[\omega,\tau]\in \mathcal{A}^{p+q}_\la{L}(X,\mathrm{End}(E))$ can be described as,
\begin{equation}\label{liealgebra}
[\omega,\tau](\xi_1,\dots,\xi_{p+q})=\frac{1}{p! q!}\displaystyle\sum_{\sigma\in S(p+q)} Sign(\sigma)[\omega(\xi_{\sigma(1)},\dots,\xi_{\sigma(p)}),\tau(\xi_{\sigma(p+1)},\dots,\xi_{\sigma_{(p+q)}})]
\end{equation}
for $\xi_1,\dots,\xi_{p+q}\in \mathcal{A}^0_\la{L}(X)$.

\subsection{Gauge orbits of Lie algebroid connections}
Consider a complex vector bundle $(E,\pi,X)$  on a complex manifold $X$, a vector bundle endomorphism is a pair $(\phi,\phi')$ such that the diagram 
 \[\begin{tikzcd}
	E & E \\
	X & X
	\arrow["\phi", from=1-1, to=1-2]
	\arrow["\pi"', from=1-1, to=2-1]
	\arrow["\pi", from=1-2, to=2-2]
	\arrow["{\phi'}"', from=2-1, to=2-2]
\end{tikzcd}\]
commutes. The endomorphism $(\phi,\phi')$ is also called a vector bundle morphism $\phi:E\rightarrow E$ over the map $\phi':X\rightarrow X$. Composition of vector bundle endomorphism $(\phi,\phi')$ and $(\psi,\psi')$ is the endomorphism $(\phi\circ \psi,\phi'\circ\psi')$, an invertible vector bundle endomorphism is called vector bundle automorphism. 
\begin{definition}\rm{
The collection of all vector bundle automorphism over the map $\textrm{id}_X:X\rightarrow X$ has the group structure with operation as composition, is called the Gauge group. The elements of the gauge group are called gauge transformations.
}
\end{definition}
The space of Gauge transformations for a vector bundle $E$ will be denoted by $\mathrm{Gau}(E)$ and a gauge transformation $(\phi,\mathrm{id}|_X)$ will be expressed by $\phi$ only.

For a given $\la{L}$-connection $\nabla\in A(E,\la{L})$ and a gauge transformation $\phi\in \mathrm{Gau}(E)$, Define the map $\nabla^\phi:\mathcal{A}^0_\la{L}(X,E)\rightarrow \mathcal{A}^1_\la{L}(X,E)$ given by
$$
\nabla^\phi(s)= (id_{\la{L}^\star}\otimes \phi^{-1})\circ \nabla (\phi(s))\quad(\text{for }s\in \mathcal{A}^0_\la{L}(X,E))
$$
\begin{lemma}
The map $\nabla^\phi:\mathcal{A}^0_\la{L}(X,E)\rightarrow \mathcal{A}^1_\la{L}(X,E)$ is a smooth Lie algebroid connection.
\end{lemma}
\begin{proof}
For $f\in \mathcal{A}^0_\la{L}(X)$ and $s\in \mathcal{A}^0_\la{L}(X,E)$, we have
\begin{align*}
\nabla^\phi(fs)=(id_{\la{L}^\star}\otimes \phi^{-1})\circ \nabla (\phi(fs))&=(id_{\la{L}^\star}\otimes \phi^{-1})\Big(d_\la{L}(f)\otimes \phi(s)+f\nabla\big(\phi(s)\big)\Big)\\
&=d_\la{L}(f)\otimes s+f(id_{\la{L}^\star}\otimes \phi^{-1})\circ\nabla\circ\phi(s)\\
&=d_\la{L}(f)\otimes s+f\nabla^\phi(s)
\end{align*}
This implies $\nabla^\phi\in A(E,\la{L})$, for given $\nabla\in A(E,\la{L})$ and $\phi\in \mathrm{Gau}(E)$.
\end{proof}
%\begin{remark}\rm{
Hence, there is a canonical left $\mathrm{Gau}(E)$ action on the space $A(E,\la{L})$, of $\la{L}$-connections, given as
\begin{equation}\label{gaugeaction1}
\odot:\mathrm{Gau}(E)\times A(E,\la{L})\quad\big(\phi,\nabla)\mapsto \phi\odot\nabla\equiv \nabla^\phi\big).
\end{equation}
Also, for a given element $\xi\in \Gamma(X,\la{L})$,
\begin{equation}\label{gaugeaction2}
(\phi,\nabla_\xi)\mapsto \phi\odot\nabla_\xi\equiv \nabla_\xi^\phi=\phi^{-1}\circ \nabla_\xi\circ \phi
\end{equation}
Since $A(E,\la{L})$ is $\mathrm{Gau}(E)$-invariant, we can talk about the quotient space $\bigslant{A(E,\la{L})}{\mathrm{Gau}(E)}$.

%}
%\end{remark}
\begin{definition}\rm{
For a given $\la{L}$-connection $\nabla$, the subgroup 
$$
\mathrm{Gau}(E)_\nabla=\{\phi\in \mathrm{Gau}(E)\mid \phi\odot\nabla\equiv \nabla\}
$$
is called the isotropic subgroup or stabiliser for $\nabla$. Note that the sugroup $\C^\star.\mathrm{id}_E\subset \mathrm{Gau}(E)$ is contained in $\mathrm{Gau}(E)_\nabla$ for each $\nabla\in A(E,\la{L})$.
}
\end{definition}
\begin{definition}\rm{
An $\la{L}$-connection $\nabla$ is called irreducible if $\mathrm{Gau}(E)_\nabla\equiv \C^\star.\mathrm{id}_E$ and reducible if it is not an irreducible $\la{L}$-connection.
}
\end{definition}
The space of irreducible $\la{L}$-connections will be denoted by $\widehat{A}(E,\la{L})$.
\begin{remark}\label{gaugeinvariant}\rm{
For $\nabla\in A(E,\la{L})$ and $\phi\in \mathrm{Gau}(E)$, let $\psi\in \mathrm{Gau}(E)_{\nabla^\phi}$ then,
\begin{align*}
(id_{\la{L}^\star}\otimes \psi^{-1})\circ \nabla^\phi(\psi(s))&\equiv \nabla^\phi(s)\\
(id_{\la{L}^\star}\otimes \psi^{-1})\circ (id_{\la{L}^\star}\otimes \phi^{-1})\circ \nabla(\phi\circ \psi(s))&\equiv (id_{\la{L}^\star}\otimes \phi^{-1})\circ \nabla(\phi(s))\\
(id_{\la{L}^\star}\otimes \phi)\circ (id_{\la{L}^\star}\otimes \psi^{-1})\circ (id_{\la{L}^\star}\otimes \phi^{-1})\circ \nabla(\phi\circ \psi\circ \phi^{-1}(s))&\equiv \nabla(s)\\
\text{or, }\big(id_{\la{L}^\star}\otimes (\phi\circ \psi\circ \phi^{-1})^{-1}\big)\circ \nabla(\phi\circ \psi\circ \phi^{-1}(s))&\equiv \nabla(s)
\end{align*}
We have $\phi\circ \psi\circ \phi^{-1}\in\mathrm{Gau}(E)_\nabla$ or, $\mathrm{Gau}(E)_{\nabla^\phi}\subset \phi^{-1}. \mathrm{Gau}(E)_\nabla. \phi$. Following the steps in reverse direction, we get the reverse inclusion i.e. $\mathrm{Gau}(E)_{\nabla^\phi}\supset \phi^{-1}. \mathrm{Gau}(E)_\nabla. \phi$. Hence,
$$
\mathrm{Gau}(E)_{\nabla^\phi}=\{\phi\circ x \circ \phi^{-1}\mid x\in \mathrm{Gau}(E)_\nabla\}
$$ 
}
\end{remark}
From above remark, if $\nabla$ is an irreducible $\la{L}$-connection, then for any $\phi\in \mathrm{Gau}(E)$ we have $\mathrm{Gau}(E)_{\nabla^\phi}=\C^\star.\mathrm{id}_E$, means the space of irreducible smooth Lie algebroid connections, $\widehat{A}(E,\la{L})$ is a $\mathrm{Gau}(E)$-invariant subset of $A(E,\la{L})$. Hence, we can talk about the quotient space $\bigslant{\widehat{A}(E,\la{L})}{\mathrm{Gau}(E)}$.

%{\color{green}In section {\color{red}to cite}, we have discussed the Sobolev completions of the quotient spaces $\bigslant{A(E,\la{L})}{\mathrm{Gau}(E)}$ and $\bigslant{\widehat{A}(E,\la{L})}{\mathrm{Gau}(E)}$ and in section {\color{red}to cite}, we have proved that {\color{red}to cite} is a pricipal {\color{red}to write} bundle.}
\begin{remark}\label{gaugeconnection}\rm{
For a given $\la{L}$-connection $\nabla\in A(E,\la{L})$, a gauge transformation $\phi\in \mathrm{Gau}(E)$ and $s\in \mathcal{A}^0_\la{L}(X,E)$, we have
\begin{align*}
\nabla^\phi(s)&=\phi^{-1}\circ (\nabla(\phi(s))\\
&=\phi^{-1}\circ (\nabla^{\mathrm{End}(E)}\phi)(s)+\phi^{-1}\circ\Big(\phi\big(\nabla(s)\big)\Big)\\
&=\phi^{-1}\circ (\nabla^{\mathrm{End}(E)}\phi)(s)+\nabla(s)
\end{align*}
and if $\nabla=\nabla_0+\alpha$ for some $\alpha\in \mathcal{A}^1_\la{L}(X,\mathrm{End}(E))$ then,
\begin{align*}
\nabla^\phi(s)&=\phi^{-1}\circ (\nabla(\phi(s))\\
&=\phi^{-1}\circ \big((\nabla_0+\alpha)(\phi(s))\big)\\
&=\phi^{-1}\circ \big(\nabla_0(\phi(s))+\alpha\circ\phi(s)\big)\\
&=\nabla_0(s)+\phi^{-1}\circ (\nabla_0^{\mathrm{End}(E)}\phi)(s)+\phi^{-1}\circ \alpha\circ \phi(s)\\
\end{align*}
Hence, $\nabla^\phi$ can be expressed as,
$$
\nabla^\phi\equiv \nabla_0+\alpha^\phi
$$
where, $\alpha^\phi=\phi^{-1}\circ (\nabla_0^{\mathrm{End}(E)}\phi)+\phi^{-1}\circ \alpha\circ \phi$
}
\end{remark}
\begin{theorem}\label{kernel}
For a given $\la{L}$-connections $\nabla$, the following statements are equivalent,
\begin{enumerate}
\item $\mathrm{Gau}(E)_\nabla=\C^\star.\mathrm{id}_E$,
\item $\mathrm{ker}(\nabla^{\mathrm{End}(E)})=\C.\mathrm{id}_E$.
\end{enumerate}
\end{theorem}
\begin{proof}
The proof follows using the following two facts:
\begin{enumerate}
\item For $\phi\in \mathrm{Gau}(E)$, $\mathrm{Gau}(E)_\nabla=\mathrm{ker}(\nabla^{\mathrm{End}(E)})$,
\item For a given vector bundle endomorphism $\phi$ for a complex vector bundle on a compact complex manifold, we can have $c\in \C$ with $|c|$ sufficiently large with $I-\frac{\phi}{c}\in \mathrm{Gau}(E)$.
\end{enumerate}
\end{proof}
\begin{definition}
The quotient space $\bigslant{\mathrm{Gau}(E)}{\C.Id_E}$ is called the reduced Gauge group.
\end{definition}
\begin{remark}\rm{
Since, $\C.Id\subset \mathrm{Gau}(E)$ acts trivially on each $\la{L}$-connection, the left $\mathrm{Gau}(E)$ action on $A(E,\la{L})$ factors through left $\mathrm{Gau}(E)^r$ action on $A(E,\la{L})$. Also, the isotropy group associated to each irreducible connection is $\C.Id_E$, the left $\mathrm{Gau}(E)^r$ action on the space of irreducible $\la{L}$-connection, $\widehat{A}(E,\la{L})$ is a free action. Hence, we have the moduli spaces
$$
B(E,\la{L})=\bigslant{A(E,\la{L})}{\mathrm{Gau}(E)^r}\text{ and }\widehat{B}(E,\la{L})=\bigslant{\widehat{A}(E,\la{L})}{\mathrm{Gau}(E)^r}
$$
which are studied in \cite{LK}.
}
\end{remark}
\section{Real (resp. quaternionic) $\la{L}$-connections and Sobolev spaces}\label{sec-3}
\begin{definition}\rm{
A toplogical complex manifold with involution is a pair $(X,\inv{X})$ of a complex manifold $X$ and an anti-holomorphic map $\inv{X}:X\rightarrow X$ such that $\inv{X}^2\equiv \textrm{id}_X$. 
}
\end{definition}
\begin{definition}\rm{
A complex vector bundle $(E,\pi,X)$ on a complex manifold with involution $(X,\inv{X})$ is said to have a real (respectively, quaternionic) structure, if there is a bundle automorphism $(\rst{E},\inv{X})$ such that 
\begin{enumerate}
%\item the diagram 
%\[\begin{tikzcd}
%	E & E \\
%	X & X
%	\arrow["{\sigma^E}", from=1-1, to=1-2]
%	\arrow["\pi"', from=1-1, to=2-1]
%	\arrow["\pi", from=1-2, to=2-2]
%	\arrow["{\sigma_X}", from=2-1, to=2-2]
%\end{tikzcd}\]
%is commutative,
\item on each fiber, the map $\rst{E}$ is $\C$-antilinear,
\item $\rst{E}\circ \rst{E}\equiv \textrm{id}_E$ (respectively, $-\textrm{id}_E$).
\end{enumerate}
A complex vector bundle with real (respectively, quaternionic) structure on a complex manifold with involution $(X,\inv{X})$, is called real(respectively, quaternionic) vector bundle, it will be denoted by $(E,\rst{E})$.
}
\end{definition}
\begin{remark}\rm{
For a given complex vector bundle $E$, there is a complex conjugate vector bundle $\overline{E}$ and can be described as a complex vector bundle in which complex structure is multiplication by $-i$ instead of $i$ also, it can be described as a bundle with same fiber as of $E$ but 1-cocycle map is complex conjugate of the 1-cocycle map of $E$.
}
\end{remark}
\begin{remark}\rm{
A real (respectively, quaternionic) structure on a smooth complex vector bundle is equivalent to having a $\C$-linear isomorphism $\Phi:E\rightarrow \overline{\invs{X}E}$ over the identity map $\mathrm{id}_X:X\rightarrow X$ along with $\Phi^{-1}\equiv \overline{\invs{X}\Phi}$ (respectively, $\Phi^{-1}\equiv -\overline{\invs{X}\Phi}$).
}
\end{remark}
\begin{remark}\rm{
For each $k>0$, a given smooth complex Lie algebroid $(\la{L},[.,.],\sharp)$ and a real (respectively, quaternionic) smooth complex bundle $(E,\rst{E})$ on a complex manifold with involution $(X,\inv{X})$ induces $\C$-linear isomorphism $\tilde{\Phi}:\wedge^k \la{L}^\star\otimes E$ and $\overline{\invs{X}(\wedge^k \la{L}^\star\otimes E)}$ covering the identity map on $X$ as well as real (resp. quaternionic) structure $(\tilde{\sigma},\sigma_X)$ on $\wedge^k \la{L}^\star\otimes E$.
}
\end{remark}
\subsection{Real structure on space of connections}\label{realstructure}
For a given smooth complex Lie algebroid $(\la{L},[.,.],\sharp,\rst{\la{L}})$ and a smooth real (resp. quaternionic) vector bundle $(E,\rst{E})$ on a complex manifold with involution $(X,\inv{X})$, let $\nabla=\nabla_0+\alpha$ be an $\la{L}$-connection, where $\alpha\in \mathcal{A}^1_\la{L}\big(X,\mathrm{End}(E)\big)$. Define
$$
\inv{A(E,\la{L})}(\nabla)=\Overline[1.3pt]{\nabla}=\nabla_0+\tilde{\Phi}^{-1}\circ \invs{X}\overline{\alpha}\circ\Phi
$$
and the associted differential operator
$$
d^{\Overline[1.3pt]{\nabla}}=d^{\nabla_0}+\tilde{\Phi}^{-1}\circ \invs{X}\overline{\alpha}\circ\Phi
$$
It is easy to verify that $\inv{A(E,\la{L})}\circ \inv{A(E,\la{L})}\equiv \mathrm{id}_{A(E,\la{L})}$, the map $\inv{A(E,\la{L})}$ is an involution on $A(E,\la{L})$.
\begin{proposition}
For a given smooth compelx Lie algebroid $(\la{L},[.,.],\sharp)$ and smooth real (resp. quaternionic) complex vector bundle $(E,\rst{E})$  on a complex manifold with involution $(X,\inv{X})$. An $\la{L}$-connection $\nabla\in A(E,\la{L})$ is a real (respectively, quaternionic) iff $\Overline[1.3pt]{\nabla}=\nabla$.
\end{proposition} 
\begin{proof}
It is an easy observation that
$$
d^{\Overline[1.3pt]{\nabla}}(\rst{E}\circ s\circ \inv{X})=\tilde{\inv{}}\circ d^\nabla( s)\circ\inv{}\quad (\text{ by definition })
$$
and an $\la{L}$-connection $\nabla$ is a real (respectively, quaternionic) connection iff 
%the following diagram commutes
%\[\begin{tikzcd}
%	{\mathcal{A}^0_\mathcal{L}(X,E)} & {\mathcal{A}^1_\mathcal{L}(X,E)} \\
%	{\mathcal{A}^0_\mathcal{L}(X,E)} & {\mathcal{A}^1_\mathcal{L}(X,E)}
%	\arrow["{d^{\nabla}}", from=1-1, to=1-2]
%	\arrow["s\mapsto \rst{E}\circ s\circ \inv{X}"', from=1-1, to=2-1]
%	\arrow["{s\mapsto \tilde{\rst{}}\circ s\circ \inv{X}}", from=1-2, to=2-2]
%	\arrow["{d^\nabla}"', from=2-1, to=2-2]
%\end{tikzcd}\]
%That is for $s\in \mathcal{A}^0_\la{L}(X,E)$, 
$$
\tilde{\inv{}}\circ d^\nabla(s)\circ \inv{}=d^\nabla(\inv{E}\circ s\circ \inv{})
$$ 
for $s\in \mathcal{A}^0_\la{L}(X,E)$. Hence, an $\la{L}$-connection $\nabla$ is a real (respectively, quaternionic) $\la{L}$-connection iff $\Overline[1.5pt]{\nabla}=\nabla$.
\end{proof}
%\begin{definition}[Real connection]
%A Lie algebroid connection $\nabla$ is called real Lie algebroid connections if $\inv{Conn}(\nabla)=\nabla$.
%\end{definition}
The space of real (respectively, quaternionic) $\la{L}$-connecitons is given by
$$
A(E,\la{L})^{\inv{}}=\{\nabla\in A(E,\la{L})\mid \Overline[1.3pt]{\nabla}=\nabla\}.
$$
The real (respectively, quaternionic) structure $\rst{E}$ on the bundle $E$ and the involution $\inv{X}$ on the manifold $X$ induces an involution $\inv{Gau(E)}:\mathrm{Gau}(E)\rightarrow \mathrm{Gau}(E)$, given by
$$
\Overline[1.3pt]{\phi}=\inv{Gau(E)}(\phi)\equiv\Phi^{-1}\circ \invs{X}\overline{\phi}\circ\Phi
$$
\begin{remark}\rm{
Note that for each $s\in \mathcal{A}^0_\la{L}(X,E)$ and $\phi\in \mathrm{Gau}(E)$,
$$
\rst{E}\circ \phi(s)\circ \inv{}=\Overline[1.3pt]{\phi}(\rst{E}\circ s\circ \inv{})\quad \text{( by definition )}
$$
and $\phi$ is a real (respectively, quaternionic) gauge transformstion iff 
$$
\rst{E}\circ \phi(s)\circ\inv{}=\phi(\rst{E}\circ s\circ \inv{})
$$
}
\end{remark}
The space of real (respectively, quaternionic) gauge transformations is given by
$$
\mathrm{Gau}(E)^{\inv{}}=\{\phi\in \mathrm{Gau}(E)\mid\Overline[1.3pt]{\phi}=\phi\}
$$

The space of real (respectively, quaternionic) Gauge transformations is a Lie group with Lie algebra 
$$
\mathfrak{gau}(E)^{\inv{}}=\mathcal{A}^0_\la{L}(X,\mathrm{End}(E))^{\inv{}}=\{u\in \mathcal{A}^0_\la{L}(X,\mathrm{End}(E))\mid \Phi^{-1}\circ \invs{X}\overline{u}\circ\Phi\equiv u\} 
$$ 
and the associated Lie braket structure comes from \eqref{liealgebra}.
%\begin{equation}\label{liebraket1}
%[\omega,\tau](\xi_1,\dots,\xi_{p+q})=\frac{1}{p!q!}\displaystyle\sum_{\sigma\in S^{p+q}}Sign(\sigma)[\omega(\xi_{\sigma(1)},\dots,\xi_{\sigma(p)}),\tau(\xi_{\sigma(p+1)},\dots,\xi_{\sigma(p+q)})
%\end{equation}
%for $\omega\in \mathcal{A}^p_\la{L}(X,\mathrm{End}(E))^{\inv{}}$, $\tau\in \mathcal{A}^q_\la{L}(X,\mathrm{End}(E))^{\inv{}}$ and $\xi_1,\dots,\xi_{p+q}\in \Gamma(X,\la{L})^{\inv{}}$
\begin{theorem}\label{gaugeinv}
For a given $\la{L}$-connection $\nabla$ and a gauge transformation $\phi\in \mathrm{Gau}(E)$, we have
$$
\Overline[1.3pt]{\nabla^\phi}=\Overline[1.3pt]{\phi}\odot\Overline[1.3pt]{\nabla}
$$
\end{theorem}
\begin{proof}
From Remark \ref{gaugeconnection}, we have 
$$
\nabla^\phi=\nabla+\mathrm{id}_{\la{L}^\star}\otimes\phi^{-1}\circ (\nabla^{\mathrm{End}(E)}\phi)
$$
Furthermore,
\begin{align*}
\Overline[1.3pt]{\nabla^\phi}&=\Overline[1.3pt]{\nabla}+\Overline[1.3pt]{\mathrm{id}_{\la{L}^\star}\otimes\phi^{-1}\circ (\nabla^{\mathrm{End}(E)}\phi)}\\
&=\Overline[1.3pt]{\nabla}+\Overline[1.3pt]{\mathrm{id}_{\la{L}^\star}\otimes\phi^{-1}}\circ \Overline[1.3pt]{(\nabla^{\mathrm{End}(E)}\phi)}\\
&=\Overline[1.3pt]{\nabla}+\mathrm{id}_{\la{L}^\star}\otimes\Overline[1.3pt]{\phi}^{-1}\circ (\Overline[1.3pt]{\nabla}^{\mathrm{End}(E)}\Overline[1.3pt]{\phi})\\
&=\Overline[1.3pt]{\phi}\odot\Overline[1.3pt]{\nabla}\\
\end{align*}
Hence, the left gauge action on the space of $\la{L}$-connections is real (respectively, quaternionic) structure invariant.
\end{proof}
\begin{remark}\rm{
Using above theorem, the space $A(E,\la{L})^{\inv{}}$ is $\mathrm{Gau}(E)^{\inv{}}$-invariant under the Gauge group action given by
\begin{equation}\label{gaugeaction3}
\mathrm{Gau}(E)^{\inv{}}\times A(E,\la{L})^{\inv{}} \rightarrow A(E,\la{L})^{\inv{}}\quad \big((\phi,\nabla)\mapsto \phi\odot \nabla\equiv \nabla^\phi\big)
\end{equation}
and we can talk about the quotient space $\bigslant{A(E,\la{L})^{\inv{}}}{\mathrm{Gau}(E)^{\inv{}}}$.
}
\end{remark}
\begin{remark}
For a given real (respectively, quaternionic) connection $\nabla\in A(E,\la{L})^{\inv{}}$, the real (respectively, quaternionic) isotropic subgroup
$$
\mathrm{Gau}(E)^{\inv{}}_\nabla=\{\phi\in \mathrm{Gau}(E)\mid \Overline[1.3pt]{\phi}=\phi\text{ and }\phi\odot\nabla\equiv \nabla^\phi\equiv \nabla\}
$$
always contains the subgroup $\mathbb{R}^\star.\mathrm{id}\subset \mathrm{Gau}(E)$.
\end{remark}
\begin{definition}[Irreducible connections]
A real (respectively, quaternionic) Lie algebroid connection $\nabla$ is called an irreducible real (respectively, quaternionic) $\la{L}$-connection if the real isotropic subgroup $\mathrm{Gau}_\nabla^{\inv{}}=\mathbb{R}^\star.\mathrm{id}_E$, otherwise the connection $\nabla$ is called reducible real (respectively, quaternionic) $\la{L}$-connection. The space of irreducible real (respectively, quaternionic) $\la{L}$-connections will be denoted by $\widehat{A}(E,\la{L})^{\inv{}}$.
\end{definition}
\begin{remark}\rm{
Following the similar line of arguements as in remark \ref{gaugeinvariant}, the space of irreducible real (respectively, quaternionic) $\la{L}$-connections $\widehat{A}(E,\la{L})^{\inv{}}$ is $\mathrm{Gau}(E)^{\inv{}}$-invariant and we can talk about the quotient $\bigslant{\widehat{A}(E,\la{L})^{\inv{}}}{\mathrm{Gau}(E)^{\inv{}}}$.
}
\end{remark}
\begin{theorem}\label{kernel1}
For a given real (respectively, quaternionic) $\la{L}$-connections $\nabla\in A(E,\la{L})^{\inv{}}$, the following statements are equivalent:
\begin{enumerate}
\item $\mathrm{Gau(E)}_\nabla=\R^\star \mathrm{id}_E$
\item $\mathrm{Ker}(\nabla^{\mathrm{End}(E)}=\R\mathrm{id}_E$
\end{enumerate}
\end{theorem}
\begin{proof}
The proof follows in the same line of arguements as in the Theorem \ref{kernel}.
\end{proof}
\begin{definition}
The quotient space $\left[\mathrm{Gau}(E)^{\inv{}}\right]^r=\bigslant{\mathrm{Gau}(E)^{\inv{}}}{\R^\star\mathrm{id}_E}$ is called reduced real Gauge group.
\end{definition}
\begin{remark}\rm{
Since $\R^\star\mathrm{id}_E\subset \mathrm{Gau}(E)^{\inv{}}_\nabla$ for each real $\la{L}$-connection, the gauge group action \ref{gaugeaction3} factors through left $\left[\mathrm{Gau}(E)^{\inv{}}\right]^r$-action and we have the quotient spaces (moduli spaces),
$$
B(E,\la{L})^{\inv{}}=\bigslant{A(E,\la{L})^{\inv{}}}{\left[\mathrm{Gau}(E)^{\inv{}}\right]^r}\text{ and }\widehat{B}(E,\la{L})^{\inv{}}=\bigslant{\widehat{A}(E,\la{L})^{\inv{}}}{\left[\mathrm{Gau}(E)^{\inv{}}\right]^r}
$$
which we have studied in Section \ref{sec-4}. We have proved $p:\widehat{A}(E,\la{L})^{\inv{}}\rightarrow\widehat{B}(E,\la{L})^{\inv{}}$ has principal $\left[\mathrm{Gau}(E)^{\inv{}}\right]^r$-bundle structure and the quotient space $\widehat{B}(E,\la{L})^{\inv{}}$ has locally Hausdorff Hilbert manifold structure.
}
\end{remark}
\begin{proposition}\label{gaugeorbitprop}
Let $\nabla,\nabla'\in A(E,\la{L})_l^{\inv{}}$ be two real (respectively, quaternionic) irreducible connections then they lie in same $\mathrm{Gau}(E)$-orbit iff they lie in same $\mathrm{Gau}(E)^{\inv{}}$-orbit.
\end{proposition}
\begin{proof}
Let $\nabla,\nabla'\in A(E,\la{L})^{\inv{}}$ be two real (respectively, quaternionic) connecitons lying in same $\mathrm{Gau}(E)^{\inv{}}$-orbit  then the correspoinding $\mathrm{Gau}(E)$-orbits will also be equal. So, only thing we need to prove is that if the corresponding $\mathrm{Gau}(E)$-orbits are same then the corresponding $\mathrm{Gau}(E)^{\inv{}}$-orbits are same.

Let $u\in \mathrm{Gau}(E)$ such that $\nabla=u\odot \nabla'$. We have,
\begin{align*}
\nabla=u\odot \nabla'=\Overline[1.3pt]{\nabla}&=\Overline[1.3pt]{u}\odot \Overline[1.3pt]{\nabla'}\\
&=\Overline[1.3pt]{u}\odot \nabla'
\end{align*} 
From above, we have $\nabla'=\big(u^{-1}.\Overline[1.3pt]{u}\big)\odot\nabla'$ but $\nabla'$ is an irreducle $\la{L}$-connection implies $u^{-1}.\Overline[1.3pt]{u}\in \C^\star\times \mathrm{id}_E$. Let $u^{-1}.\Overline[1.3pt]{u}=re^{i\theta}\mathrm{id}_E$ and $u'=\frac{u}{\sqrt{r}}$. Note that, $\nabla=u'\odot \nabla'$ and $u'^{-1}.\Overline[1.3pt]{u'}=e^{i\theta}\mathrm{id}_E$. Define $v=e^{i\frac{\theta}{2}}u'$, then 
\begin{align*}
\Overline[1.3pt]{v}&=e^{-i\frac{\theta}{2}}\Overline[1.3pt]{u'}\\
&=e^{i\frac{\theta}{2}}u'\\
&=v
\end{align*}
Note that $v\odot \nabla'=u'\odot \nabla'=\nabla$, implies $\nabla$ and $\nabla'$ are in same $\mathrm{Gau}(E)^{\inv{}}$-orbit.
\end{proof}
\subsection{Sobolev completion of the space of connections}
For a given smooth complex Lie algebroid $(\la{L},[.,.],\sharp)$ and smooth complex vector bundle $E$ on a complex manifold $X$, let $h^\la{L}$ and $h^E$ be the Hermitian metrics on the bundles, respectively. Let $g$ be the Riemannian metric on the complex manifold, which further gives the volume density $\mathrm{vol}_g$ as well as Borel measure $\mu_g$ on $X$. We can define the Sobolev space of $\la{L}$-connections, described as below.

A global section $\psi\in \Gamma(X,E)$ is called Borel measurable, if $\psi^{-1}(U)\subset X$ is Borel measurable, for each open subset $U\subset E$. Let $L^2_l(X,E)$, $l\in\mathbb{N}-\{0\}$ be the vector space of global Borel measurable sections whose weak derivatives of order upto $l$ are square integrable. The Sobolev spaces $L^2_l(X,E)$ are Hilbert spaces with inner product
$$
\langle \psi,\phi\rangle_{L^2_l}=\displaystyle\sum_{j=0}^l\langle \nabla^j\psi,\nabla^j\phi\rangle_{L^2}=\displaystyle\sum_{j=0}^l\int \langle \nabla^j\psi,\nabla^j\phi\rangle_{h^E}d\mu_g
$$
where $\langle \nabla^j\psi,\nabla^j\phi\rangle_{L^2}$ is the finite $L^2$-inner product of weak derivatives of order $j$ of given global sections $\phi,\psi\in \Gamma(X,E)$.

The metric $h^E$ on $E$ induces a metric on $\mathrm{End}(E)$ given by $\langle f_1,f_2\rangle_{h^{\mathrm{End}(E)}}=\mathrm{tr}(f_2^\star.f_1)$, where $f_2^\star$ is adjoint of $f_2$ w.r.t metric $h^E$. Using the Hermitian metric $h^{\la{L}}$ and the induced Hermitian metric $h^{\mathrm{End}(E)}$, the Sobolev space $L^2_l(X,\wedge^k\la{L}^\star\otimes\mathrm{End}(E))$ which will be denoted by $\mathcal{A}^k_\la{L}\big(X,\mathrm{End}(E)\big)_l$, can be defined.

Using the induced Hermitian metric $h^{\mathrm{End}(E)}$, we have $L^2$-metric on the space $\mathcal{A}^0_\la{L}\big(X,\mathrm{End}(E)\big)^{\inv{}}$ as well as the following $L^2$-orthogonal decomposition
\begin{equation}\label{decomp}
\mathcal{A}^0_\la{L}\big(X,\mathrm{End}(E)\big)^{\inv{}}=\Big[\mathcal{A}^0_\la{L}\big(X,\mathrm{End}(E)\big)^{\inv{}}\Big]^0+\mathbb{R}.\mathrm{id}_E
\end{equation}
where, $\Big[\mathcal{A}^0_\la{L}\big(X,\mathrm{End}(E)\big)^{\inv{}}\Big]^0=\left\{f\in \mathcal{A}^0_\la{L}\big(X,\mathrm{End}(E)\big)^{\inv{}}\mid \int \mathrm{tr}(f)d\mu_g=0\right\}$
%The Hermitian metrices $h_{\mathrm{End}(E)}$ and $h_{\Lambda^k\la{L}^\star}$ on the bundles $\mathrm{End}(E)$ and $\Lambda^k\la{L}^\star$, respectively induces the Sobolev space $L^2_l\big(X,\Lambda^k\la{L}^\star\otimes \mathrm{End}(E)\big)$ known as Sobolev completion of the space $\mathcal{A}^k_\la{L}\big(X,\mathrm{End}(E)\big)$ denoted by $\mathcal{A}^k_\la{L}\big(X,\mathrm{End}(E)\big)_l$

%{\color{red} for two elements $f_1,f_2\in \Gamma(X,\mathrm{End}(E))$, $\langle f_1,f_2\rangle\in \Gamma(X,\underline{\C})$. Furthermore, the scalar product of the elements $f_1,f_2$ is given by,
%\begin{equation}\label{innerbundle}
%(f_1,f_2)=\displaystyle\int \langle f_1,f_2 \rangle d\mu^\ol_g
%\end{equation}
%}
%Define the subspace $\mathcal{A}^0_\la{L}(X,\mathrm{End}(E))^0\subset \mathcal{A}^0_\la{L}(X,\mathrm{End}(E))$ as
%$$
%\mathcal{A}^0_\la{L}(X,\mathrm{End}(E))^0=\{f\in \mathcal{A}^0_\la{L}(X,\mathrm{End}(E))\mid (f,id_E)= 0\}
%$$
%It is easy to verify that,
%$$
%\mathcal{A}^0_\la{L}(X,\mathrm{End}(E))=\mathcal{A}^0_\la{L}(X,\mathrm{End}(E))^0\bigoplus \underline{\dC} \cdot \mathrm{id}_E
%$$
%The above decomposition is orthogonal decomposition w.r.t the inner product \eqref{innerbundle}.

Using lemma \ref{liealgconn}, for a fixed trivial smooth lie algebroid connections $\nabla_0$, the Sobolev space of connections is given by
$$
A(E,\la{L})_l=\{\nabla_0+\alpha \mid \alpha\in \mathcal{A}^1_\la{L}(X,\mathrm{End}(E))_l\}
$$
The real (respectively, quaternionic) structure $\rst{E}$ induces an involution $\inv{A(E,\la{L})_l}$, which can be described in same line of arguments as in Subsection \ref{realstructure}.
\begin{definition}
For a given smooth real (respectively, quaternionic) vector bundle $E$ and a smooth complex Lie algebroid $(\la{L},[.,.],\sharp)$, a Sobolev $\la{L}$-connection $\nabla\in A(E,\la{L})_l$ is called a real (respectively, quaternionic) Sobolev $\la{L}$-connection, if $\inv{A(E,\la{L})_l}(\nabla)=\Overline[1.3pt]{\nabla}=\nabla$.
\end{definition}
The space of real (respectively, quaternionic) Sobolev $\la{L}$-connections is given by
$$
A(E,\la{L})_l^{\inv{}}=\{\nabla\in A(E,\la{L})_l\mid \Overline[1.3pt]{\nabla}=\nabla\}=\{\nabla_0+\alpha \mid \alpha\in \mathcal{A}^1_\la{L}(X,\mathrm{End}(E))_l^{\inv{}}\}
$$
Thus, the space of real Sobolev connections is a Hilbert space as well as a Hilbert manifold.

\subsection{Sobolev completion of space of gauge transformations}
To define the Sobolev gauge group action on the Sobolev space of connections, we need to define Sobolev completion of the gauge group $\mathrm{Gau}(E)_{l+1}$, in case of $l> \frac{1}{2}\mathrm{dim}_{\mathbb{R}}X$.

For $l> \frac{1}{2}\mathrm{dim}_{\mathbb{R}}X$, using Sobolev embedding theorem,  we have the compact embedding 
$$
\mathcal{A}^0_\la{L}(X,\mathrm{End}(E))_{l+1}\subset \mathcal{A}^0_\la{L}(X,\mathrm{End}(E))_l\subset C^0_\la{L}(X,\mathrm{End}(E))
$$ 
along with,
$$
\mathcal{A}^0_\la{L}(X,\mathrm{End}(E))_{l+1}^\sigma\subset \mathcal{A}^0_\la{L}(X,\mathrm{End}(E))_l^\sigma\subset C^0_\la{L}(X,\mathrm{End}(E))^{\inv{}}
$$ 
where $C^0_\la{L}(X,\mathrm{End}(E))$ is the space of global continuous sections.

Using Sobolev multiplication theorem, the multiplication 
$$
\mathcal{A}^0_\la{L}(X,\mathrm{End}(E))\times \mathcal{A}^0_\la{L}(X,\mathrm{End}(E))\rightarrow \mathcal{A}^0_\la{L}(X,\mathrm{End}(E))\quad\big((\phi,\psi)\mapsto \phi\circ \psi\big)
$$ 
can be extended continuously to 
$$
\mathcal{A}^0_\la{L}(X,\mathrm{End}(E))_{l+1}^\sigma\times \mathcal{A}^0_\la{L}(X,\mathrm{End}(E))_{l+1}^\sigma\rightarrow \mathcal{A}^0_\la{L}(X,\mathrm{End}(E))_{l+1}^\sigma.
$$
This continuous extension makes $\mathcal{A}^0_\la{L}(X,\mathrm{End}(E))_{l+1}^\sigma$, a Banach algebra over $\R$.

Define the gauge group,
$$
\mathrm{Gau}(E)_{l+1}^\sigma=\{\phi\in \mathcal{A}^0_\la{L}(X,\mathrm{End}(E))_{l+1}^\sigma\mid \phi \text{ is invertible in }\mathcal{A}^0_\la{L}(X,\mathrm{End}(E))_{l+1}^\sigma\}
$$
which is an open subset of $\mathcal{A}^0_\la{L}(X,\mathrm{End}(E))_{l+1}^\sigma$; hence $\mathrm{Gau}(E)_{l+1}^\sigma$ is a Hilbert manifold. Furthermore, $\mathrm{Gau}(E)_{l+1}^\sigma$ is a Hilbert Lie group with Lie algebra $\mathfrak{gau}(E)_{l+1}^\sigma=\mathcal{A}^0_\la{L}(X,\mathrm{End}(E))_{l+1}^\sigma$ and the Lie bracket structure can be described by continuously extending the Lie bracket structure \eqref{liealgebra} to appropriate Sobolev spaces of real (respectively, quaternionic) bundle morphisms, using Sobolev multiplication theorem in the range $l>\frac{1}{2}\mathrm{dim}_{\R}(X)$.

In the range $l>\frac{1}{2}\mathrm{dim}_{\R}(X)$, using the compact embedding $\mathcal{A}^0_\la{L}\big(X,\mathrm{End}(E)\big)_{l+1}^\sigma\subset \mathcal{A}^0_\la{L}\big(X,\mathrm{End}(E)\big)_l^\sigma$, the Sobolev space $\mathcal{A}^1_\la{L}\big(X,\mathrm{End}(E)\big)_l^\sigma$ has topological right as well as left $\mathcal{A}^0_\la{L}\big(X,\mathrm{End}(E)\big)_{l+1}^\sigma$ module structure. 

The Gauge group action \ref{gaugeaction3} can be extended continuously, using the bimodule structure and Sobolev multiplication theorem, described as
\begin{align*}
\odot:\mathrm{Gau}(E)_{l+1}^{\inv{}}\times A(E,\la{L})^{\inv{}}_l&\rightarrow A(E,\la{L})^{\inv{}}_l\\
(\phi,\nabla)&\mapsto \phi\odot \nabla=\nabla^\phi\\
&\qquad\qquad\hspace{4pt}=\nabla_0+\alpha^\phi=\nabla_0+\phi^{-1}\circ (\nabla^{\mathrm{End}(E)}_0\phi)+\phi^{-1}\circ \alpha\circ \phi
\end{align*}
where, $\nabla_0^{\mathrm{End}(E)}:\mathcal{A}^0_\la{L}(X,\mathrm{End}(E))\rightarrow \mathcal{A}^1_\la{L}(X,\mathrm{End}(E))$ is a linear operator. Since continuous linear operators are smooth and matrix multiplication is also smooth, the gauge action $\odot$ is a smooth map between Hilbert manifolds.

We can define the space $\widehat{A}(E,\la{L})^{\inv{}}_l$ of irreducible Sobolev real (resp. quaternionic) connections as done in Subsection \ref{realstructure} and following the same line of arguements as in \ref{gaugeinvariant}, it can be verified that $\widehat{A}(E,\la{L})^{\inv{}}_l \subset A(E,\la{L})^{\inv{}}_l$ is $\mathrm{Gau}(E)^{\inv{}}_{l+1}$-invariant subspace. Define the moduli space,
$$
B(E,\la{L})^{\inv{}}_l=\bigslant{A(E,\la{L})^{\inv{}}_l}{\mathrm{Gau}(E)^{\inv{}}_{l+1}}
$$
of Sobolev real (respectively, quaternionic) connections as well as modulie space 
$$
\widehat{B}(E,\la{L})^{\inv{}}_l=\bigslant{\widehat{A}(E,\la{L})^{\inv{}}_l}{\mathrm{Gau}(E)^{\inv{}}_{l+1}}
$$
of irreducible Sobolev real (respectively, quaternionic) connections.

\begin{theorem}
Let $G$ be a Banach-Lie group over a field $\R$ with Lie algebra $\mathfrak{g}$ and $N$ a normal Banach-Lie subgroup over field $\R$ of $G$ then the quotient space $\bigslant{G}{N}$ has Lie group structure with Lie algebra $\bigslant{\mathfrak{g}}{\mathfrak{n}}$ in a unique way, such that the quotient map $q:G\rightarrow \bigslant{G}{N}$ is a smooth map and any map between Banach manifolds $f:\bigslant{G}{N}\rightarrow X$ between Banach manifolds is smooth iff the map $f\circ q:G\rightarrow X$ is smooth. 
\end{theorem}
\begin{proof}
See \cite{MJDJFG},\cite{GH} and \cite{HGKHN}
\end{proof}
From above theorem, the reduced real Sobolev Gauge group 
$$
(\mathrm{Gau}(E)_{l+1}^{\inv{}})^r=\bigslant{\mathrm{Gau}(E)_{l+1}^{\inv{}}}{R^\star . \mathrm{id}_E}
$$ 
is a Hilbert Lie group with Lie algebra $\Big[\mathcal{A}^0_\la{L}\big(X,\mathrm{End}(E)\big)^{\inv{}}_{l+1}\Big]^0$

\subsection{Tangential map associated to Gauge action and its adjoint} Let $\nabla=\nabla_0+\alpha$ be a fixed connection in $A(E,\la{L})^{\inv{}}_l$ and the associated differential operator is $d^\nabla$. Note that $$
d^\nabla\big(\mathcal{A}_\la{L}^0(X,\mathrm{End}(E))^{\inv{}}_{l+1}\big)\subset \mathcal{A}_\la{L}^1(X,\mathrm{End}(E))^{\inv{}}_l
$$ 
Let $i:\mathcal{A}^0_\la{L}(X,\mathrm{End}(E))^{\inv{}}_{l+1}\hookrightarrow \mathcal{A}^0_\la{L}(X,\mathrm{End}(E))^{\inv{}}_l$ be the canonical compact embedding, the tangential map associated to Gauge action at $(\mathrm{id}_E,\nabla)$ is given by
$$
d^\nabla:\mathcal{A}^0_\la{L}(X,\mathrm{End}(E))_{l+1}^{\inv{}}\rightarrow \mathcal{A}^1_\la{L}(x,\mathrm{End}(E))_l^{\inv{}}\quad \big(d^\nabla(\gamma)=d^{\nabla_0}(\gamma)+[\alpha,\gamma]\big)
$$
The tangential map $d^\nabla$ can be described as $d^\nabla\equiv d^{\nabla_0}+\mathrm{ad}(\alpha)\circ i$, where 
$$
\mathrm{ad}(\alpha):\mathcal{A}^0_\la{L}(X,\mathrm{End}(E))^{\inv{}}_l\rightarrow \mathcal{A}^1_\la{L}(X,\mathrm{End}(E))^{\inv{}}_l\quad\big(\mathrm{ad}(\alpha)(\gamma)\mapsto [\alpha,\gamma]\big)
$$ 
can be yielded by continuous extension of the map
$$
m:\mathcal{A}^1_\la{L}(X,\mathrm{End}(E))^{\inv{}}\times \mathcal{A}^0_\la{L}(X,\mathrm{End}(E))^{\inv{}}\rightarrow \mathcal{A}^1_\la{L}(X,\mathrm{End}(E))^{\inv{}}\qquad\big((\alpha,\gamma)\mapsto [\alpha,\gamma]\big)
$$ 
to appropriate Sobolev spaces, using Sobolev mutiplication thorem.
%, the continuous extension $m_l$ is given by
%$$
%m_l:\mathcal{A}^1_\la{L}(X,\mathrm{End}(E))_l\times \mathcal{A}^0_\la{L}(X,\mathrm{End}(E))_l\rightarrow \mathcal{A}^1_\la{L}(X,\mathrm{End}(E))_l \quad\left((\alpha,\gamma)\mapsto [\alpha,\gamma]\right)
%$$

The formal adjoint operator associated to 
$$
[\alpha,.]:\mathcal{A}^0_\la{L}(X,\mathrm{End}(E))^{\inv{}}\rightarrow \mathcal{A}^1_\la{L}(X,\mathrm{End}(E))^{\inv{}}\quad\big(\text{for }\alpha\in \mathcal{A}^1_\la{L}(X,\mathrm{End}(E))^{\inv{}}\big)
$$ 
w.r.t Hermitian metric $h^{\mathrm{End}(E)}$ gives a $C^\infty(X,\mathbb{R})$-bilinear map
$$
m^\star:\mathcal{A}^1_\la{L}(X,\mathrm{End}(E))^{\inv{}}\times \mathcal{A}^1_\la{L}(X,\mathrm{End}(E))^{\inv{}}\rightarrow \mathcal{A}^0_\la{L}(X,\mathrm{End}(E))^{\inv{}}
$$.
Using Sobolev multiplication theorem and continuosly extending the map $m^\star$ to appropriate Sobolev spaces, we have the map
$$
\mathrm{ad}(\alpha)^\star:\mathcal{A}^1_\la{L}(X,\mathrm{End}(E))^{\inv{}}_l\rightarrow \mathcal{A}^0_\la{L}(X,\mathrm{End}(E))^{\inv{}}_l
$$
for each fixed $\alpha\in \mathcal{A}^1_\la{L}(X,\mathrm{End}(E))^{\inv{}}_l$ and the adjoint operator 
$$
(d^\nabla)^\star:\mathcal{A}^1_\la{L}(X,\mathrm{End}(E))^{\inv{}}_l\rightarrow \mathcal{A}^0_\la{L}(X,\mathrm{End}(E))^{\inv{}}_{l-1}
$$ 
which can be described as $(d^\nabla)^\star=(d^{\nabla_0})^\star+i\circ\mathrm{ad}(\alpha)^\star$. The operator $(d^{\nabla_0})^\star$ is the formal adjoint of the operator $d^{\nabla_0}$ w.r.t induced Hermitian metric on $\la{L}^\star \otimes \mathrm{End}(E)$. 
%where $[\mathrm{ad}(\alpha)]^\star:\mathcal{A}^1_\la{L}(X,E)_l\rightarrow \mathcal{A}^0_\la{L}(X,E)_l$ can be yielded by using Sobolev multiplication theorem and continuously extending the operator 
%$$
%m^\star:\mathcal{A}^1_\la{L}(X,\mathrm{End}(E))\times \mathcal{A}^0_\la{L}(X,\mathrm{End}(E))\rightarrow \mathcal{A}^1_\la{L}(X,\mathrm{End}(E))\qquad\big((\alpha,\gamma)\mapsto \mathrm{ad}(\alpha)^\star(\gamma)\big)
%$$ 
%The formal adjoint operator $\mathrm{ad}(\alpha)^\star:\Omega^1_\la{L}(X,\mathrm{End}(E))\rightarrow \Omega^0_\la{L}(X,\mathrm{End}(E))$ associated to the differential operator $\mathrm{ad}(\alpha)$ can be yielded by continuously extending the adjoint operator associated to {\color{red}to write}, using Sobolve multiplication theorem 
%$$
%m':\Omega^1_\la{L}(X,\mathrm{End}(E))_l\times \Omega^1_\la{L}(X,\mathrm{End}(E))_l\rightarrow \Omega^0_\la{L}(X,\mathrm{End}(E))_l\quad \Big((\alpha,\beta)\mapsto \mathrm{ad}(\alpha)^\star(\beta)\Big)
%$$  
%which is complex anti-linear in first co-ordinate while complex linear in second co-ordinate. The operator $m'$ further gives the continuous operator $\mathrm{ad}(\alpha)^\star:\Omega^1_\la{L}(X,\mathrm{End}(E))_l\rightarrow \Omega^0_\la{L}(X,\mathrm{End}(E))_l$ and the adjoint differential operator $(d^\nabla)^\star$ can be expressed as
%$$
%(d^\nabla)^\star\equiv (d^{\nabla_0})^\star+i\circ (\mathrm{ad}(\alpha))^\star
%$$

\section{Moduli of real (respectively, quaternionic) Lie algebroid connections}\label{sec-4}
In this section, we will assume the the smooth complex Lie algebroid $(\la{L},[.,.],\sharp)$ is a transitive Lie algebroid. For a given smooth complex vector bundle with real (respectively, quaternionic) structure $(E,\rst{E})$, the moduli space as gauge equivalence classes, which we want to study are $B(E,\la{L})^{\inv{}}$ and $\widehat{B}(E,\la{L})^{\inv{}}$ also, we have described the Hilbert manifold structure on $\widehat{B}(E,\la{L})^{\inv{}}$ and proved that $\widehat{p}:A(E,\la{L})_l^{\inv{}}\rightarrow \widehat{B}(E,\la{L})^{\inv{}}$ is a principal $\mathrm{Gau}(E)^{\inv{}}_{l+1}$ bundle.  

%\begin{definition}[Elliptic differential operator]\rm{
%A differential operator $L\in Diff_m(E,F)$ is called an elliptic differential operator if the symbol map $\sigma_m(L)(x,\nu)\in \mathrm{Hom}_\C(E_x,F_x)$ is an isomorphism for each $(x,\nu)\in T_xX$.
%}
%\end{definition}
For any connection $\nabla=\nabla_0+\alpha$ in $A(E,\la{L})_l$, the associated Laplacian 
$$
\Delta_\alpha=(d^\nabla)^\star\circ d^\nabla:\mathcal{A}^0_\la{L}(X,\mathrm{End}(E))_{l+1}\rightarrow \mathcal{A}^0_\la{L}(X,\mathrm{End}(E))_{l-1}
$$ 
can be expressed as, 
$$
\Delta_\alpha=\Delta_0+d^{\nabla_0}\circ \mathrm{ad}(\alpha)\circ i+i\circ \big(\mathrm{ad}(\alpha)\big)^\star\circ \mathrm{ad}(\alpha)\circ i+i\circ \big(\mathrm{ad}(\alpha)\big)^\star \circ d^{\nabla_0}
$$ 
Since continuous operators on Banach spaces are compact and composition of compact and continuous operators on Sobolev space is compact, the Laplacian $\Delta_\alpha$ will be a Fredholm operator if the Laplacian $\Delta_0$ is a Fredholm operator. But the symbol map ${\inv{}}_1(d^{\nabla_0})(\xi_x)(.)\equiv \sharp^\star(\xi_x)\otimes(.)$ and ${\inv{}}_1\big((d^{\nabla_0})^\star\big)(\xi_x)(.)\equiv \big({\inv{}}_1(d^{\nabla_0})(\xi_x)\big)^\star(.)\equiv \big(\sharp^\star(\xi_x)\big)^\star\otimes(.)$ implies the symbol map ${\inv{}}_2\big((d^{\nabla_0})^\star\circ d^{\nabla_0}\big)(\xi_x)$ is an isomorphism iff the dual map $\sharp^\star:(T^{\C} X)^\star\rightarrow \la{L}^\star$ has trivial kernel or, the anchor map $\sharp:\la{L}\rightarrow T^{\C}X$ is surjective.

Hence, the Laplacian $\Delta_\alpha=(d^\nabla)^\star\circ d^\nabla$ for a given connection $\nabla=\nabla_0+\alpha$ in $A(E,\la{L})_l$ is an Elliptic differential operator furthermore a Fredholm operator (for details, see \cite[Chapter 4]{ROW}.

%Note that, for a given real (respectively, quaternionic) connection $\nabla=\nabla_0+\alpha$,
%$$
%\Delta_\alpha(\mathcal{A}^0_\la{L}(X,\mathrm{End}(E))^{\inv{}}_{l+1})\subset \mathcal{A}^0_\la{L}(X,\mathrm{End}(E))^{\inv{}}_{l-1}
%$$

\begin{lemma}\label{decomp1}
For any $\nabla\in A(E,\la{L})$ and $l>\frac{1}{2}\mathrm{dim}_{\mathbb{R}}X$, we have the $L^2$-orthogonal decomposition
$$
\mathcal{A}^1_\la{L}(X,E)_l=(\mathrm{im}(d^\nabla))\oplus (\mathrm{ker}(d^\nabla)^\star)
$$
\end{lemma}
\begin{proof}
From above discussion, the Laplacian $\Delta_\alpha$ associated to the differential $d^\nabla$ is Fredholm, the $\mathrm{Ker}(\Delta_\alpha)$ will be finite dimensional and $\mathrm{im}(\Delta_\alpha)$ will be a closed subspace. 

Let $\mathrm{Ker}(\Delta_\alpha)^\perp$ is the $L^2$-orthogonal complement then using Banach open mapping theorem for the bijective continuous map $\Delta_\alpha|_{\mathrm{Ker}(\Delta_\alpha)}$, the map $\big(\Delta_\alpha|_{\mathrm{Ker}(\Delta_\alpha)}\big)^{-1}$ is continuous. Also, using the fact that $\mathrm{Ker}(\Delta_\alpha)=\mathrm{Ker}(d_\nabla)$, $d^\nabla|_{\mathrm{Ker}(\Delta_\alpha)^\perp}=\mathrm{im}(d^\nabla)$. But $\mathrm{im}(d^\nabla)$ lies in the subspace $\big((d^\nabla)^\star\big)^{-1}\big(\mathrm{im}(\Delta_\alpha)\big)$, which is kernel of the continuous map $\mathrm{id}_X-d^\nabla\circ \Delta_\alpha|_{\mathrm{Ker}(\Delta_\alpha)^\perp}\circ (d^\nabla)^\star$ on $X= \big((d^\nabla)^\star\big)^{-1}(\mathrm{im}(\Delta_\alpha)$. We have, $\mathrm{im}(d^\nabla)$ is closed subspace of $\mathcal{A}^1_\la{L}(X,\mathrm{End}(E))$, furthermore 
%$\left[\mathrm{im}(d^\nabla)\right]^{\inv{}}$ will be closed subspace of $\mathcal{A}^1_\la{L}(X,\mathrm{End}(E))^{\inv{}}$. 
we have the $L^2$ orthogonal decomposition,
$$
\mathcal{A}^1_\la{L}(X,\mathrm{End}(E))_l=\mathrm{im}(d^\nabla)\oplus \big(\mathrm{im}(d^\nabla)\big)^\perp
$$
But using the property of adjoint operator $\langle d^\nabla(\alpha),\beta\rangle_{L^2}=\langle\alpha,(d^\nabla)^\star(\beta)\rangle_{L^2}$, we have $\big(\mathrm{im}(d^\nabla)\big)^\perp=\mathrm{ker}(d^\nabla)^\star$ and the required $L^2$-orthogonal decomposition
$$
\mathcal{A}^1_\la{L}(X,\mathrm{End}(E))_l=\mathrm{im}(d^\nabla)\oplus \mathrm{ker}(d^\nabla)^\star
$$
\end{proof}
\begin{lemma}\label{irredopen}
The space of real (respectively, quaternionic) irreducible connections $\widehat{A}(E,\la{L})_l^{\inv{}}\subset A(E,\la{L})_l^{\inv{}}$ is an open subset, for $l>\frac{1}{2}\mathrm{dim}(M)$
\end{lemma}
\begin{proof}
Let $\mathcal{F}(\mathcal{A}^0_{\la{L},l+1},\mathcal{A}^0_{\la{L},l-1})$ be the collection of all Fredholm operators between the Sobolev spaces $\mathcal{A}^0_\la{L}(X,\mathrm{End}(E))_{l+1}$ and $\mathcal{A}^0_\la{L}(X,\mathrm{End}(E))_{l-1}$, the composition of maps
$$
A(E,\la{L})_l\xrightarrow{\nabla\mapsto \Delta_\alpha} \mathcal{F}(\mathcal{A}^0_{\la{L},l+1},\mathcal{A}^0_{\la{L},l-1})\xrightarrow{\Delta_\alpha\mapsto \mathrm{dim}\textrm{ ker}(\Delta_\alpha)} \mathbb{R}
$$
is upper semi-continuous (see \cite{CJADSD}). Note that, for a gauge transformation $\phi$, 
$$
\phi\odot \nabla=\nabla\iff d^\nabla(\phi)=0\quad (\text{see theorem }\ref{kernel})
$$
and $\mathrm{Ker}\big((d^\nabla)^\star\circ d^\nabla)\big)=\mathrm{Ker}(d^\nabla)$ along with $\mathrm{dim}(d^\nabla)\geq 1$ implies that $\widehat{A}(E,\la{L})_l$ is an open subset of $A(E,\la{L})$. Furthermore, the fixed point set $\widehat{A}(E,\la{L})^{\inv{}}_l$ is an open subset of Hilbert space $A(E,\la{L})_l^{\inv{}}$ (fixed points set of $A(E,\la{L})_l$ with induced topology).
\end{proof}
\begin{remark}
Since $B(E,\la{L})_l^{\inv{}}$ has quotient topology induced from the topology of $A(E,\la{L})^{\inv{}}_l$ and $\widehat{B}(E,\la{L})^{\inv{}}_l=p(\widehat{A}(E,\la{L})^{\inv{}}_l)$, using above theorem the quotient space $\widehat{B}(E,\la{L})^{\inv{}}_l$ is open in $B(E,\la{L})_l^{\inv{}}$.
\end{remark}
Consider a fixed Sobolev real (respectively, quaternionic) $\la{L}$-connection $\nabla\in A(E,\la{L})_l^{\inv{}}$ and the Hilbert submanifold of $A(E,\la{L})_l^{\inv{}}$,
$$
\mathcal{O}_{\nabla,\epsilon}^{\inv{}}=\{\nabla+\alpha\mid \alpha\in \mathcal{A}^1_\la{L}(X,\mathrm{End}(E))_l^{\inv{}}, (d^\nabla)^\star(\alpha)=0\text{ and }||\alpha||_{L^2_l}<\epsilon\}
$$
The tangent space is $T_\nabla(\mathcal{O}_{\nabla,\epsilon}^{\inv{}})=(\mathrm{ker}(d^\nabla)^\star)^{\inv{}}$. By taking $\epsilon$ small enough and using lemma \ref{irredopen}, we can assume that $\mathcal{O}_{\nabla,\epsilon}^{\inv{}}\subset \widehat{A}(E,\la{L})_l^{\inv{}}$.

%-------------------------------------------
\begin{theorem}[Main theorem]
The space $\widehat{B}(E,\la{L})^{\inv{}}_l$ is a locally Hausdorff Hilbert manifold and $\widehat{p}:\widehat{A}(E,\la{L})^{\inv{}}_l\rightarrow \widehat{B}(B,\la{L})^{\inv{}}_l$ is a principal $\mathrm{Gau}^r_{l+1}$ bundle for $l>\frac{1}{2}\mathrm{dim}_\mathbb{R}X$.
\end{theorem}
\begin{proof}
The proof is given in steps as follows.
\begin{enumerate}
\item \textbf{Claim:} The map 
$$
\Psi_\nabla^{\inv{}}:\left[\mathrm{Gau}(E)^{\inv{}}_{l+1}\right]^r\times \mathcal{O}_{\nabla,\epsilon}^{\inv{}}\rightarrow \widehat{A}(E,h)^{\inv{}}_l
$$
given by $\Psi_{\nabla}^{\inv{}}(\phi,\nabla+\alpha)=\phi\odot(\nabla+\alpha)$ is smooth and a local diffeomorphism.

\item[]\textbf{Proof:} In \cite{LK}, author has proved, the map
$$
\Psi_\nabla:\mathrm{Gau}(E)_{l+1}^r\times \mathcal{O}_{\nabla,\epsilon}\rightarrow \widehat{A}(E,\la{L})_l
$$
a local diffeomorphism that is, for $\epsilon$ small enough there are nbd $N_{\mathrm{id}_E}$ of $\mathrm{id}_E\in \mathrm{Gau}(E)^r_{l+1}$ and $\mathcal{U}_\nabla$ of $\nabla\in A(E,\la{L})_l$ such that the map $\Psi_\nabla:N_{\mathrm{id}_E}\times \mathcal{O}_{\nabla,\epsilon}\rightarrow \mathcal{U}_\nabla$ is diffeomorphism.

Let $\Psi_\nabla^{\inv{}}|_{N_{\mathrm{id}_E}^{\inv{}}\times \mathcal{O}_{\nabla,\epsilon}^{\inv{}}}\equiv \Psi_\nabla|_{N_{\mathrm{id}_E}^{\inv{}}\times \mathcal{O}_{\nabla,\epsilon}^{\inv{}}}$ then $\Psi_{\nabla}^{\inv{}}$ will be bijective using Theorem \ref{gaugeinv} and will be diffeomorphism because $\Psi_\nabla|_{N_{\mathrm{id}_E}^{\inv{}}\times \mathcal{O}_{\nabla,\epsilon}^{\inv{}}}$ is diffeomorphism.
\item \textbf{Claim:} For $\epsilon>0$ small enough, the map $p_{\nabla,\epsilon}^{\inv{}}=p|_{\mathcal{O}^{\inv{}}_{\nabla,\epsilon}}:\mathcal{O}_{\nabla,\epsilon}^{\inv{}}\rightarrow \widehat{B}(E,\la{L})_l$ is injective.
\item[]\textbf{Proof:}Using Proposition \ref{gaugeorbitprop} and injectivity of $p_{\nabla,\epsilon}$ (see \cite{LK}), we have the injectivity of $p_{\nabla,\epsilon}^{\inv{}}$.
\item \textbf{Claim:}The subset $\mathcal{U}^{\inv{}}_{\nabla,\epsilon}=p^{\inv{}}_{\nabla,\epsilon}(\mathcal{O}^{\inv{}}_{\nabla,\epsilon})$ is an open subset in $\widehat{B}(E,\la{L})_l^{\inv{}}$ and the map
$$
\Psi^{\inv{}}_\nabla:\mathrm{Gau}(E)_{l+1}^{\inv{}}\times \mathcal{O}_{\nabla,\epsilon}^{\inv{}}\rightarrow \left[p^{\inv{}}_{\nabla,\epsilon}\right]^{-1}(\mathcal{U}^{\inv{}}_{\nabla,\epsilon})
$$
is diffeomorphism.
\item[]\textbf{Proof:} For any $\phi\in \mathrm{Gau}(E)_{l+1}^{\inv{}}$, take a nbd $\mathcal{W}=L_\phi(\mathcal{N}_{\mathrm{id}_E})$, where $L_\phi$ is left translation by $\phi\in \mathrm{Gau}(E)_{l+1}^{\inv{}}$. Then,
$$
\Psi|_{\mathcal{W}\times \mathcal{O}_{\nabla,\epsilon}^{\inv{}}}\equiv \tilde{\lambda}_\phi\circ \Psi|_{\mathcal{N}_{\mathrm{id}_E}\times \mathcal{O}_{\nabla,\epsilon}^{\inv{}}}\circ \big(L_{\phi^{-1}}\times \mathrm{id}_{\widehat{A}(E,\la{L})^{\inv{}}_l}\big)|_{\mathcal{W}\times \mathcal{O}_{\nabla,\epsilon}^{\inv{}}}
$$
is diffeomorphism, where $\tilde{\lambda}_\phi:\widehat{A}(E,h)^{\inv{}}_l\rightarrow \widehat{A}(E,h)^{\inv{}}_l$ is left multiplication by $\phi$.
\item For any $\nabla'\in A(E,\la{L})_l^{\inv{}}$ such that $p(\nabla')\in \mathcal{U}_{\nabla,\epsilon}^{\inv{}}$, define 
$$
g_\nabla(\nabla')=\mathrm{pr}_1\circ \Psi_\nabla^{-1}(\nabla')
$$
Now, define the chart map $\sigma_\nabla:\mathcal{U}_{\nabla,\epsilon}^{\inv{}}\rightarrow \mathcal{O}_{\nabla,\epsilon}^{\inv{}}$ satisfying,
$$
\sigma_\nabla(p^{\inv{}}_{\nabla,\epsilon}(\nabla'))=\left[g_\nabla(\nabla')\right]^{-1}\odot \nabla'
$$
for $\nabla'\in p_{\nabla,\epsilon}^{\inv{}}(\mathcal{U}_{\nabla,\epsilon}^{\inv{}})$.

\item[]\textbf{Claim:} The Hilbert manifold $\widehat{B}(E,\la{L})_l^{\inv{}}$ is a smooth Hilbert manifold.
\item[]\textbf{Proof:} We need to verify, for two charts $\mathcal{U}_{\nabla,\epsilon}^{\inv{}}$ and $\mathcal{U}_{\nabla',\epsilon'}^{\inv{}}$ the transition map 
$$
\sigma_\nabla\circ \sigma_{\nabla'}^{-1}:\sigma_{\nabla'}\big(\mathcal{U}_{\nabla',\epsilon'}\cap \mathcal{U}_{\nabla,\epsilon}\big)\rightarrow \sigma_{\nabla}\big(\mathcal{U}_{\nabla',\epsilon'}\cap \mathcal{U}_{\nabla,\epsilon}\big)
$$ 
is smooth.
Take $\nabla'+\alpha\in \mathcal{O}_{\nabla',\epsilon'}^{\inv{}}$ such that $p_{\nabla',\epsilon'}^{\inv{}}(\nabla'+\alpha)\in \mathcal{U}_{\nabla',\epsilon'}\cap \mathcal{U}_{\nabla,\epsilon}$, we have
$$
\sigma_\nabla\circ \sigma_{\nabla'}^{-1}(\nabla'+\alpha)=\sigma_\nabla(p^{\inv{}}_{\nabla,\epsilon}(\nabla'+\alpha)\big)=\left[g_\nabla(\nabla'+\alpha)\right]^{-1}\odot (\nabla'+\alpha)
$$
Hence, the transition map is smooth as gauge action is smooth and the map $g_\nabla$ is smooth.
\end{enumerate}
\end{proof}
%----------------------------------------------------------------------------

%%%%%%%%%%%%%%%%%%%%%%%%%%%%%%%%%%%%%%%%%%%%%%%%%%%%%%%%%%%%%%%%%%%%%%%
\end{document}